\documentclass[letterpaper]{article}

\usepackage[english]{babel}
\usepackage[utf8x]{inputenc}
\usepackage[T1]{fontenc}

\usepackage[a4paper,top=3cm,bottom=2cm,left=3cm,right=3cm,marginparwidth=1.75cm]{geometry}

\usepackage{graphicx}
\usepackage{url}
\usepackage{amsmath,amsfonts,amssymb,color}
\usepackage[pagebackref=true,colorlinks]{hyperref}
\usepackage{mathptmx}     
\usepackage{amsmath,amsfonts,amsthm,amssymb,color}
\usepackage[pagebackref=true,colorlinks]{hyperref}
\usepackage{comment}

\newcommand\numberthis{\addtocounter{equation}{1}\tag{\theequation}}

\newtheorem{theorem}{Theorem}[section]
\newtheorem{corollary}[theorem]{Corollary}
\newtheorem{lemma}[theorem]{Lemma}

\newtheorem{proposition}[theorem]{Proposition}
\newtheorem{claim}[theorem]{Claim}

\theoremstyle{definition}
\newtheorem{definition}[theorem]{Definition}
\newtheorem{remark}[theorem]{Remark}

\begin{document}

\title{Diagonal Form of the Varchenko Matrices
}

\author{Yibo Gao, YiYu Zhang
}

\date{}

\maketitle

\begin{abstract}
Varchenko \cite{varchenko1993bilinear} defined the Varchenko matrix associated to any real hyperplane arrangement and computed its determinant. In this paper, we show that the Varchenko matrix of a hyperplane arrangement has a diagonal form if and only if it is semigeneral, i.e., without degeneracy. In the case of semigeneral arrangement, we present an explicit computation of the diagonal form via combinatorial arguments and matrix operations, thus giving a combinatorial interpretation of the diagonal entries.

\end{abstract}

\section{Introduction}
\label{intro}
Varchenko defined the Varchenko matrix associated with any real hyperplane arrangement in \cite{varchenko1993bilinear} and computed its determinant, which has a very nice factorization. Naturally, one may ask about its Smith normal form or diagonal form over some integer polynomial ring.  The Smith normal forms of the $q$-Varchenko matrices for certain types of hyperplane arrangements were first studied by Denham and Hanlon in \cite{denham1997smith} and more recently by Cai and Mu in \cite{caimu}.

\

In this paper, we prove that the Varchenko matrix of a real hyperplane arrangement has a diagonal form if and only if the arrangement is semigeneral. 
We define hyperplane arrangements and the associated Varchenko matrices in section 2. In section 3, we use combinatorial arguments and matrix operations to explicitly construct a diagonal form of the Varchenko matrix associated with a semigeneral hyperplane arrangement, therefore giving a combinatorial interpretation of the diagonal form.  Finally, we prove by contradiction that the Varchenko matrix of any arrangement with degeneracy does not have a diagonal form in section 4.

\

It follows immediately that while the $q$-Varchenko matrix of any semigeneral arrangement has a Smith normal form, the corresponding Varchenko matrix doesn't in general. Besides, our construction serves as an alternative proof for a special case, i.e., that of real semigeneral hyperplane arrangements, of Varchenko's theorem on the determinant of the Varchenko matrix. 

\section{Preliminaries}
In this paper, we mostly follow the notation in  \cite{stanley2004introduction}. We only consider real, finite, affine hyperplane arrangements $\mathcal{A}=\{H_1,H_2,\ldots,H_n \}$ in $\mathbb{R}^d$. 

\subsection{Hyperplane Arrangement}
First we briefly go over the notation and basic constructions in hyperplane arrangements.

\

For any subset $B\subseteq I=\{1,2,\ldots,N\}$, denote by $H_B=\bigcap_{a\in B}H_a$ the intersection of hyperplanes with index in $B$ . If $B=\emptyset$, then $H_B=\mathbb{R}^d$ by convention.
\begin{definition}
We say that $\mathcal{A}$ is a \textit{general} arrangement (or $\mathcal{A}$ is in \textit{general position}) in $\mathbb{R}^d$ if for any subset $B\subseteq I$ , the cardinality $ |B|\leq d$ implies that $ \dim(H_B)=d-|B|$, while $ |B|>d\text{ implies that } H_B=\emptyset.$

If for all $B\subseteq I$ with $H_B\neq\emptyset$, we have $ \dim{(H_B)}=d-|B|$, then $\mathcal{A}$ is called \textit{semigeneral} (or $\mathcal{A}$ is in \textit{semigeneral position}) in $\mathbb{R}^d$.
\end{definition}

It is clear that all general arrangements are semigeneral arrangements.
  
\

Let $L(\mathcal{A})$ be the partially ordered set whose elements are all nonempty intersections of hyperplanes in $\mathcal{A}$, including $\mathbb{R}^d$ as the intersection over the empty set, with partial order reverse inclusion. We call $L(\mathcal{A})$ the \textit{intersection poset} of the arrangment $\mathcal{A}$. Notice that the minimum element in $L(\mathcal{A})$ is $\mathbb{R}^d$.

Define a \textit{region} $R$ of $\mathcal{A}$ to be a connected component of the complement of $\bigcup_{a\in I}{H_a}$ in $\mathbb{R}^d$. Denote by $\mathcal{R}(\mathcal{A})$ the set of regions of $\mathcal{A}$ and $r(\mathcal{A})=|\mathcal{R}(\mathcal{A})|$ the number of regions. It is well known that if $\mathcal{A}$ is semigeneral, then $r(\mathcal{A})=|L(\mathcal{A})|.$

\subsection{The Varchenko Matrix}

Let $\mathcal{A}=\{H_1,\ldots,H_N\}$ be a real, finite hyperplane arrangement. Assign to each $H_a \in \mathcal{A}$ an indeterminate (or weight) $x_a$. For any pair of regions $(R_i,\ R_j)$ ($i,j$ not necessarily distinct) of $\mathcal{A}$, set $$\mathrm{sep}(R_i,R_j):=\{ H_a\in \mathcal{A}:H_a \mathrm{\ separates\ } R_i \mathrm {\ and\ } R_j\}.$$

To each element $M\in L(\mathcal{A})$, we assign the weight $\displaystyle{x_M=\prod_{M\subseteq H_a}{x_a}}.$ If $M=\mathbb{R}^d$, then $x_M=1$.

\begin{definition}
The \textit{Varchenko matrix} $V(\mathcal{A})=[V_{ij}]$ of a hyperplane arrangement $\mathcal{A}$ is the $r(\mathcal{A})\times r(\mathcal{A})$ matrix with rows and columns indexed by $\mathcal{R}(\mathcal{A})$ and entries $$V_{ij}=\prod_{H_a\in \mathrm{sep}(R_i,R_j)}x_a.$$  If $ \mathrm{sep}(R_i,R_j)=\emptyset$, then $V_{ij}=1$.
\end{definition}

For example, the Varchenko matrix of the arrangement in Fig. \ref{fig:figvar} is  $$V=\begin{bmatrix}1&x_1&x_1x_2&x_1x_3&x_3&x_2x_3&x_1x_2x_3\\x_1&1&x_2&x_3&x_1x_3&x_1x_2x_3&x_2x_3\\x_1x_2&x_2&1&x_2x_3&x_1x_2x_3&x_1x_3&x_3\\x_1x_3&x_3&x_2x_3&1&x_1&x_1x_2&x_2\\x_3&x_1x_3&x_1x_2x_3&x_1&1&x_2&x_1x_2\\x_2x_3&x_1x_2x_3&x_1x_3&x_1x_2&x_2&1&x_1\\x_1x_2x_3&x_2x_3&x_3&x_2&x_1x_2&x_1&1\end{bmatrix}$$. 

It turns out that the determinant of the Varchenko matrix has an elegant factorization. We formulate this result via M\"{o}bius functions as in \cite[sec 6]{stanley2004introduction}.
\begin{definition}
The \textit{M\"{o}bius function} of $L(\mathcal{A})$ is defined by 
$$\mu(M,M') = \left\{ \begin{array}{rcl}
1\ \ \ \ \ \ \ \ \ \ \ \ \ \ \ \ \ \ \ \ \ \ \ \ \  \ \ \  & \mbox{if}
& M=M' \mathrm{\ in\ }L(\mathcal{A})\\
\displaystyle{-\sum_{M\leq N< M'}{\mu(M,N)}} & \mbox{if} & M<M'\mathrm{\ in\ }L(\mathcal{A})
\end{array}\right..$$
Furthermore, we set $\mu(M)=\mu(\mathbb{R}^d,M)$.

Define the \textit{characteristic polynomial} of $\mathcal{A}$ to be $\displaystyle{\chi_{\mathcal{A}}{(t)}=\sum_{M\in L(\mathcal{A})}{\mu(M)t^{\mathrm{dim}(M)}}.}$
\end{definition}

Now we can formulate the result of Varchenko.
  
\begin{theorem}[Varchenko \cite{varchenko1993bilinear}] \label{det}
Let $\mathcal{A}=\{H_1,\ldots, H_N\}$ be a real, finite, affine hyperplane arrangement.  For all elements $M\in L(\mathcal{A})$, define the subarrangement $\mathcal{A}_M:=\{H\in\mathcal{A}:M\subseteq H\}$ and the arrangement $\mathcal{A}^M:=\{M\cap H \neq \emptyset: H\in \mathcal{A}-\mathcal{A}_M\}\subseteq M$. Set $$\  n(M)=r(\mathcal{A}^M) \textrm{\ and\ }p(M)=|\frac{d}{dt}\chi_{\mathcal{A}_M}(1)|.$$
Then the determinant of the Varchenko matrix associated with $\mathcal{A}$ is given by
$$\det V(\mathcal{A})=\prod_{M\in L(\mathcal{A}), M\neq\mathbb{R}^d}{(1-x^2_M)^{n(M)p(M)}}.$$
\end{theorem}

\begin{figure}[ht]
\centering
\includegraphics[scale=0.65]{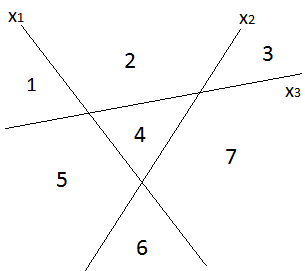}
\caption{General arrangement of three lines in  $\mathbb{R}^2$.}
\label{fig:figvar}
\end{figure}
  
For instance, the Varchenko matrix associated with the arrangement in Figure \ref{fig:figvar} has determinant $$\det(V)=(1-x^2_1)^3 (1-x^2_2)^3(1-x^2_3)^3.$$

For $R_i,R_j,R_k\in\mathcal{R}(\mathcal{A})$, where $i,j,k$ are not necessarily distinct, we define the \textit{distance}  between $R_i$ and $R_j \cup R_k$, denoted by $l_i(j,k)$, to be the product of the indeterminates $x_a$ of all hyperplanes $H_a$ that separate both $R_i, R_j$ and $R_i, R_k$.

It follows that $$l_i(j,k)=\prod_{H_a\in \mathrm{sep}(R_i,R_j)\cap \mathrm{sep}(R_i,R_k)}x_a.$$

Observe that by definition of $V(\mathcal{A})$, 
 the entry $\displaystyle{V_{ij}=\frac{V_{ij}\cdot V_{ik}}{l_i(j,k)^2}.}$

\subsection{Diagonal Form}
\label{section 1.4}

Let  $A,\ B$ be  $N\times N$ square matrices over  $\mathbb{Z}[x_1,x_2,\ldots,x_N]$.

\begin{definition} \label{def:sim}
We say that the square matrix $A$ is \textit{equivalent} to the square matrix $B$ over the ring $R$, denoted by $A\sim B$, if there exist matrices $P, Q$ over $R$ such that $\det(P), \det(Q)$ are units in $R$ and $PAQ=B$. 
\end{definition}
In other words, the matrix $A$ is equivalent to $B$ if and only if we can get from $A$ to $B$ by a series of row and column operations (subtracting a multiple of a row/column from another row/column, or multiplying a row/column by a unit in $R$). It is easy to check that $\sim$ is an equivalence relation.

For all $k\leq N$, let $\gcd(A,k)$ be the greatest common divisor of all the determinants of $k\times k$ submatrices of $A$.

\begin{lemma}\label{lemmasimilar1}
If $A\sim B$, then $\gcd(A,k)=\gcd(B,k)$ for all $k=1,2,\dots,N$, and $\mathrm{rank}(A)=\mathrm{rank}(B).$
\end{lemma}
\begin{proof}
It suffices to check the lemma when $B$ can be obtained from $A$ by one single row (or analogously, column) operation, and assume without loss of generality, that it is adding the first row multiplied by $r\in R$ to the second row. For a matrix $M$, denote the matrix obtained from $M$ by choosing rows indexed by $I$ and columns indexed by $J$ as $M_{I,J}$. Let $d=gcd(A,k)$ and we now show that every $k\times k$ submatrix $B_{I,J}$ of $B$ has determinant divisible by $d$. If $2\notin I$, then $B_{I,J}=A_{I,J}$, which has determinant divisible by $d$. If $1,2\in I$, then $B_{I,J}$ can be obtained from $A_{I,J}$ by a single row operation, implying that $\det(B_{I,J})=\det(A_{I,J})$, which is divisible by $d$. The last case is that $2\in A$ and $1\notin A$. Let $I'=(I\setminus\{2\})\cup\{1\}$. It is clear that $\det(B_{I,J})=\det(A_{I,J})+r\det(A_{I',J})$, which is again divisible by $d$. The above arguments showed that $\gcd(A,k)|\gcd(B,k)$ so by symmetry, we have the desired equality. It is a standard fact that the rank of a matrix doesn't change after row and column operations.
\end{proof}
  
\begin{definition}
Let $A$ be an $N\times N$ square matrix over the ring $R$. We say that $A$ has a \textit{diagonal form} over $R$ if there exists a diagonal matrix $D=\mathrm{diag}{(d_1,d_2,\ldots,d_N)}$ in $R$ such that $A\sim D$. In particular, if $d_i | d_{i+1}$ for all $1 \leq i \leq N-1$, then we call $D$ the \textit{Smith normal form} (SNF) of $A$ in $R$.
\end{definition}

It is known that the SNF of a matrix exists and is unique if we are working over a principal ideal domain. But the SNF of a matrix may not exist if we are working over $R$, the ring of integer polynomials. For example, the matrix $\begin{bmatrix}x&0\\0&x+2\end{bmatrix}$ does not have an SNF over $R$.

\begin{lemma}
If the SNF of a matrix $A$ exists, then it is unique up to units. 
\end{lemma}
\begin{proof}
Let $D$ be one of the SNFs of $A$. Suppose that $A\sim D=\mathrm{diag}{(d_1,\ldots,d_N)}$ where $d_k|d_{k+1}$ for $ k=1,\ldots,N-1$.

It is easy to see that $\mathrm{gcd}{(D,k)}=d_1\cdots d_k$, so $d_1\cdots d_k=\mathrm{gcd}{(A,k)}$ for $ k=1,\ldots,N-1$  by Lemma \ref{lemmasimilar1}.
Given a matrix $A$, the above equations and the condition that $d_k|d_{k+1}$ for $k=1,\ldots,N-1$ are sufficient to solve for $d_k$. Namely, $d_k=0$ if $\mathrm{gcd}{(A,k)}=0$; otherwise $d_k$ equals a unit times $\mathrm{gcd}{(A,k)}/\mathrm{gcd}{(A,k-1)}$. Here, $\mathrm{gcd}{(A,0)}=1$. 
\end{proof}

The next lemma follows  directly from the transitivity of $\sim$ and the uniqueness of the SNF.

\begin{lemma}\label{lemmasimilar2}
If $A\sim B$ and if one of $A$, $B$ has an SNF, then the other also has an SNF and $\mathrm{SNF}(A)=\mathrm{SNF}(B).$
\end{lemma}
\begin{proof}
By transitivity of the equivalence relation $\sim$, the SNF of matrix $A$ is also an SNF of $B$. The lemma follows from the uniqueness of the SNF, we obtain the desired lemma.
\end{proof}

\begin{definition}
Let $A$ be a matrix over the ring $\mathbb{Z}[x_1,x_2,\ldots,x_N]$.

Define $A_{x_1=f_1(q),x_2=f_2(q),\ldots,x_N=f_N(q)}$ to be the matrix over the ring $\mathbb{Z}[q]$ obtained by replacing each $x_i$ by $f_i(q)$ in $A$.
\end{definition}
  
For example, when $V$ is a Varchenko matrix, the matrix $V_{x=q,\ldots,x_N=q}$ is called the \textit{q-Varchenko matrix}.
  
\begin{lemma}
Let $A,B$ be matrices over the ring $\mathbb{Z}[x_1,x_2,\ldots,x_N]$. If $A\sim B$, then $$A_{x_1=f_1(q),x_2=f_2(q),\ldots,x_N=f_N(q)}\sim B_{x_1=f_1(q),x_2=f_2(q),\ldots,x_N=f_N(q)}.$$
\end{lemma}
  
\section{The Main Result}
\label{section 2}
\begin{theorem} \label{main}
Let $\mathcal{A}=\{H_1,\ldots,H_N\}$ be a real, finite, affine hyperplane arrangement in $\mathbb{R}^d$. Assign an indeterminate $x_a$ to each $H_a$, $a\in I=\{1,\ldots,N\}$. Then the Varchenko matrix $V$ associated with $\mathcal{A}$ has a diagonal form over $\mathbb{Z}[x_1,\ldots,x_N]$ if and only if $\mathcal{A}$ is in semigeneral position. In that case, the diagonal entries of the diagonal form of $V$ are exactly the products $$\prod_{a\in B}(1-x^2_a)$$  ranging over all $B\subseteq I$ such that $H_B\in L(\mathcal{A}).$
\end{theorem}

\begin{corollary}
Let $\mathcal{A}$ be any semigeneral hyperplane arrangement in $\mathbb{R}^d$. The $q$-Varchenko matrix $V_q$ of $\mathcal{A}$ has an SNF over the ring $\mathbb{Z}[q]$. The diagonal entries of its SNF are of the form $(1-q^2)^k,\ k=0,1,\ldots,d$, and the multiplicity of $(1-q^2)^k$ equals the number of elements in $L(\mathcal{A})$ with dimension $d-k$.
\end{corollary}

\begin{corollary}
Let $\mathcal{A}$ be a semigeneral hyperplane arrangement in $\mathbb{R}^d$ and $V$ its Varchenko matrix. Then $$\det(V)=\prod_{a\in I}{(1-x^2_a)}^{m_a},\ where\  m_a=|\{H_B\in L(\mathcal{A}): H_B\subseteq H_a\}|.$$
\end{corollary}

Thus, our proof also serves as an alternative proof for a special case of Theorem \ref{det}.

\section{Construction of the Diagonal Form of the Varchenko Matrices of Semigeneral Arrangements}  
\label{section 3}
  
\ \ \ \ In this section, we prove the sufficient condition of Theorem \ref{main} by explicitly constructing the diagonal form of the Varchenko matrix of a semigeneral arrangement. The construction relies on the existence of a ``good'' indexing of the set of regions (Lemma \ref{lemmaexist}), whose proof is rather technical.
  
\ 
 
Assume as before that we are working in $\mathbb{R}^d$.

\begin{definition}
A set of regions $\mathcal{B}\subset\mathcal{R}(\mathcal{A})$ \textit{encompasses a point} $x\in \mathbb{R}^d$ if the interior of the closure of the union of these regions contains $x$.

A set of regions $\mathcal{B}\subset\mathcal{R}(\mathcal{A})$ \textit{encompasses an element} $M\in L(\mathcal{A})$ if there exists a point $x\in M$ such that $\mathcal{B}$ encompasses $x$.
\end{definition}

In other words, an element $M$ is encompassed by a set of regions $\mathcal{B}$ if a nontrivial part of $M$ with nonzero relative measure is encompassed by some regions in $\mathcal{B}.$

\

Denote by $ E(\mathcal{B})$ the set of elements of the intersection poset that are encompassed by $\mathcal{B}$. Note that $ E(\emptyset)=\emptyset$ and $ E(\{R\})=\{\mathbb{R}^d\}$ for any $R\in\mathcal{R}(\mathcal A).$
  
  In Figure \ref{fig:fig1}, for example, all points on the segment of $H_3$ between region $R_3$ and $R_4$ are encompassed by the set of regions $\{R_1,R_2,R_3,R_4,R_5\}$. So $ E(\{R_1,R_2,R_3,R_4,R_5\})=\{\mathbb{R}^2,H_1,H_2,H_3,H_5\}$ and $ E(\{R_1,R_2,R_3,R_4,R_5, R_6\}) =\{\mathbb{R}^2,H_1,H_2,H_3,H_5,H_3\cap H_5\}$.

\begin{figure}[h]
\centering
\includegraphics[scale=0.4]{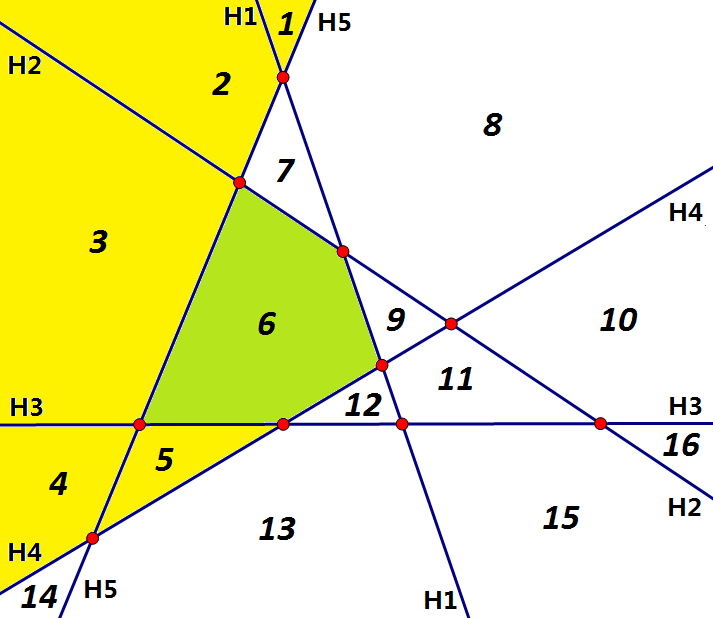}
\caption{General arrangement of 5 lines in $\mathbb{R}^2$; region $R_i$ is labeled as $i$.}
\label{fig:fig1}
\end{figure}

\begin{definition}
Fix a numbering of the regions of $\mathcal{A}$ by $1, 2, \ldots, r(\mathcal{A})$.  We say that a region $R_n$ is the \textit{first to encompass} $M$ for some $M\in L(\mathcal{A})$ if $M\in E(\{R_1,R_2,\ldots,R_n\})$ and $M\notin E(\{R_1,R_2,\ldots,R_{n-1}\})$.
\end{definition}
  
 One can see that $\gcd(V_{mn},x_{a_1}x_{a_2}\cdots x_{a_s})\neq1$ for any $1\leq m<n$, where $R_n$ is the first to encompass  $M=H_{a_1}\cap\cdots\cap H_{a_s}\in L(\mathcal{A}).$
  
For example, in Figure \ref{fig:fig1}, we see that region $R_5$ is the \textit{first to encompass} $H_5$ and region $R_6$ is the \textit{first to encompass} $H_3\cap H_5$.
  
\begin{lemma} \label{lemmaexist}
We can number the regions of $\mathcal{A}$ by $I=\{1,2,\ldots,r(\mathcal{A})\}$ such that $| E(\{R_1,R_2,\ldots,R_n\})|= n$ for all $n\in I$. Under such numbering, let $\mathcal{B}^{(n)}=\{R_1,R_2,\ldots,R_n\}$ be the set of regions with the first $n$ indices and $\mathcal{B}^{(0)}=\emptyset$. Set $E_n= E(\mathcal{B}^{(n)})\backslash E(\mathcal{B}^{(n-1)})$. The following properties hold  for all $n\in I$:
  
(a) The interior of the closure of $\bigcup_{1\leq m\leq n} R_m$  is connected. 
  
(b) For all $M\in L(\mathcal{A})$, the subset $\{x:x\in M,\ \mathcal{B}^{(n)}\ \mathrm{encompasses }\ x\}\subseteq M$ is connected.
  
(c) If $R_n$ is the \textit{first to encompass} $M=H_B
$ where $ B \subseteq I$, then $R_n$ is the first indexed region in the cone formed by all $H_a, a\in B$ that contains $R_n$. 

(d) For all $M\in L(\mathcal{A})$ is cut into connected closed sections $M_1, M_2,\ldots$ by the hyperplanes that intersect $M$. Denote by $R^{(n)}_{M_i}=\{R_m: 1\leq m\leq n,\ M_i\in\overline{R_m}\}$ the set of regions whose boundary contain $M_i$. Then the interior of the closure of $\bigcup_{R_m\in R^{(n)}_{M_i}} R_m$ is connected.
\end{lemma}
  
Lemma \ref{lemmaexist} is saying that there is a way to index the regions of $\mathcal{R}(\mathcal{A})$ one by one in numerical order such that whenever we index a new region, the set of indexed regions encompass exactly one new element in $L(\mathcal{A})$.

The labeling of the regions in Figure \ref{fig:fig1} is such an indexing. The interior of the closure of $\{R_1,R_2,R_3,R_4,R_5\}$ and $\{R_1,R_2,R_3,R_4,R_5,R_6\}$ are connected. If we add $R_6$ to the closure of $\{R_1,R_2,R_3,R_4,R_5\}$, $R_6$ is the region with the smallest index in the cone formed by $H_3$ and $ H_5$. Property (d) is saying that the interior of the closure of all indexed regions around any intersection point, line segment or ray is connected.
  
\begin{proof} [Proof of Lemma \ref{lemmaexist}]
We prove Lemma \ref{lemmaexist} by induction on $n$, the number of regions indexed. 

The base case $n=1$ is trivial since we encompass exactly $\mathbb{R}^d$ after indexing the first region $R_1$. Assign orientations to the pair of half-spaces determined by each hyperplane as follows: suppose that a hyperplane $H$ is first encompassed by $R_n$, then the half-space containing $\mathcal{B}^{(n-1)}$ is labeled as $H^0$ and the half-space containing $R_n$ is labeled as $H^-$.

Suppose Lemma \ref{lemmaexist} holds after indexing the first $n-1$ regions. Let $M=H_{a_1}\cap\cdots\cap H_{a_s}$ be an element of the smallest dimension $d-s$ in $L(\mathcal{A})$ satisfying the condition that there is an unindexed region $R\in \mathcal{R}(\mathcal{A})\backslash\mathcal{B}^{(n-1)}$ such that  $\{R\}\cup  \mathcal{B}^{(n-1)} $ encompasses $M$. Index $R$ by $R_n$.

\begin{claim} 
$ |E_n|=| E(\mathcal{B}^{(n)})|-| E(\mathcal{B}^{(n-1)})|\leq 1$. 
For $n\geq 2$, if $ |E_n|=1,$ then:

(i) $M\in E_n$ is some hyperplane $H\in\mathcal{A}$ for $s=1$;

(ii) $M\in E_n$ has $\dim(M)=d-s$ for $s\geq 2$.

Furthermore, the four properties in Lemma \ref{lemmaexist} remain true.
\end{claim}
In other words, after indexing a new region $R_n$, we encompass at most one new element in $ L(\mathcal{A})$.

\begin{proof}[Proof of Claim]
\

\noindent\textbf{Case (i)}: $s=1$. The closure of any unindexed region is connected to at most one indexed region by some hyperplane since $d-s$ is minimal. Suppose that $R_n=R$ is connected by $H_{a}$ to some $R_m,$ where $ 1\leq m\leq n-1$. Hence we definitely have $H_{a}\in  E(\mathcal{B}^{(n)})$. Note that $R_n$ is connected to only one indexed region, so $|E_n|\leq 1$.  Therefore $ E_n=\{H_{a}\}$ or $E_n=\emptyset$.

Then we want to show that the four properties still hold after adding $R_n$. Property (a) remains true by the induction hypothesis on $n-1$ regions, since $H_a$ connects $R_n$ and $\mathcal{B}^{(n-1)}$. 

For property (c): If $ E_n=\{H_{a}\}$, then $R_n$ has to be the only indexed region in $H^-_{a}$. It follows that $H_{a}\notin  E(\mathcal{B}^{(n-1)})$ and $\{x:x\in H_{a},\ \mathcal{B}^{(n)}\ \mathrm{encompasses }\ x\}=\overline{R_n}\cap\overline{R_m}$ is connected and nonempty. Therefore (c) holds for $n$. 

Note that $M'\in  E(\mathcal{B}^{(n-1)})$ implies $M'\subsetneq H_a$, so $\{x:x\in M',\ \mathcal{B}^{(n)}\ \mathrm{encompasses }\ x\}=\{x:x\in M',\ \mathcal{B}^{(n-1)}$ encompasses $ x\}$ is connected by induction hypothesis (b). If $M'\subset H_a$, then $E_n=\{H_a\}$ and $M'$ is clearly connected. This proves property (b) for $n$.

For part (d), note that $R_n$ is the only indexed region in $H^-_{a}$ and $R_{M_i}$ is changed after we add $R_n$  only if $M_i\in \overline{R_n}$.  The case where $R_n$ is the only indexed region with $M'_i$ as a supporting face  is trivial. Suppose that $M'_i\subseteq (\bigcup^n_{k=1}\overline{R_K})\cap H_a$, in which case $M'_i\subseteq \overline{R_n}\cap\overline{R_m}$. Now $\bigcup_{R\in R^{(n-1)}_{M'_i}} R$ is connected by the induction hypothesis, and $R_n, R_m$ are connected by a nontrivial codimension one subset of $H_{a}$, so $\bigcup_{R\in R^{(n)}_{M'_i}} R$ is also connected.  Therefore property (d) holds for $n$.

\

\noindent\textbf{Case (ii):} $s\geq2$. By induction hypothesis (d) for $n-1$, there exists $2^s-1$ indexed regions in $\mathcal{B}^{(n-1)}$ which, together with $R_n=R$, encompass $M=H_{1}\cap\cdots\cap H_{s}$. Hence $H_{1},\ldots,H_{s}\in E(\mathcal{B}^{(n-1)})$ and $M\in  E(\mathcal{B}^{(n)})$. For all $i=1,2,\ldots,s$, there is a unique indexed region $R_{m_i}$ such that $H_{i}$ connects $R_n$ and $ R_{m_i}$, i.e., $\mathrm{sep}(R_n,R_{m_i})=\{H_i\}$. It is clear that $1\leq m_1\neq\cdots \neq m_s\leq n-1$.

Assume to the contrary that we encompass at least two distinct elements $M$, $M'\in L(\mathcal{A})$ by adding $R_n$. If $M'$ is some hyperplane $H'\in\mathcal{A}$, then  $R_{m_1},\ldots,R_{m_s}$ would be on the same side of $H'$ as $R_n$ since $R_{m_i}$ and $R_n$ is separated by exactly one hyperplane, which is $H_i$. By the induction hypothesis (c), this implies that $H'$ has been encompassed before adding $R_n$, a contradiction to our assumption. 

Hence $M'$ has to be the intersection of some $H'_1,\ldots, H'_t \in E(\mathcal{B}^{(n-1)})$ with $ 1< t\leq d$, all of which are supporting hyperplanes of $R_n$. Denote the $t$ indexed regions connected to $R_n$ by nontrivial codimension-$s$ subsets of $M'$ as $R'_1,\ldots,R'_{2^t-1}$. Let $K'_t=\bigcap_{k=1}^t {H'_k}^{\epsilon'_k}$, where $\epsilon'_k=+$ if $R_n\subset{H'_k}^+$ and $\epsilon'_k=-$ if $R_n\subset{H'_k}^-$. Hence $K'_t$ is a polyhedron containing $R_n$ with $H'_1,\ldots, H'_t$ as facets and $M'$ as its ``tip''. Similarly, let $K_s=\bigcap_{i=1}^s H_{i}^{\epsilon_i}$, where $\epsilon_i=+$ if $R_n\subset H_{i}^+$ and $\epsilon_i=-$ if $R_n\subset H_{i}^-$. 

\begin{figure}[h]
\centering
\includegraphics[scale=0.55]{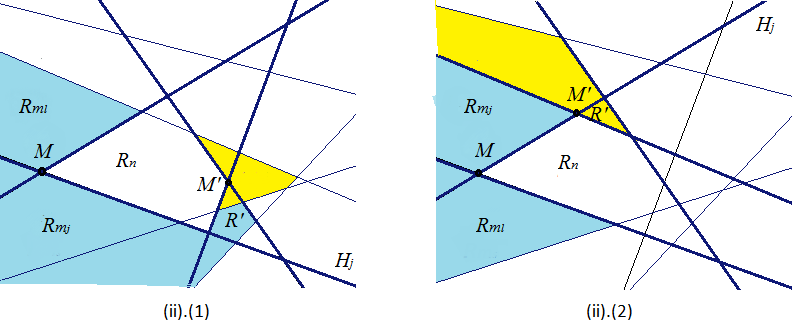}
\caption{Examples of case (ii) in $\mathbb{R}^2$.}
\label{fig:lemma}
\end{figure}

\textbf{Subcase (1).} If $\{H'_k, 1\leq k\leq t\}\cap\{H_{i}, 1\leq i\leq s\}=\emptyset$, then a nontrivial part of $M'$  is inside the interior of $K_s$ and a nontrivial part of $M$ is inside the interior of $K'_t$ by convexity. Pick any $R_{m_j}$ as defined above. Without loss of generality, suppose that $R_{m_j}\subset H_j^+$. Then $M'\subset \overline{R_n}\subset H_j^-$ and  $R'_k\subset H_j^-$ for all $k$. Using the induction hypothesis (a),  $R_{m_j}$ and $\bigcup^{2^t-1}_{k=1}{R'_k}$ have to be connected via a shortest sequence of indexed regions $R_{a_1}=R'_1,\ldots,R_{a_u}=R_{m_j}$ that lie on both sides of $H_{j}$ before $R_n$ is indexed. Let $R'=R_{a_r}$ be one of these regions such that $R'\subset H_{j}^-$ has $H_{j}$ as a supporting hyperplane and $R_{a_r+1}\subset H_j^+$. Pick any $l\neq j$ such that $R_{m_l}$ and $R'$ are on different sides of exactly one $H_i$, with $1\leq i\neq j\leq s$.  In the inductive case of $n-1$, $R_n$ is an unindexed region between the disjoint regions $R'$ and $R_{m_l}$. (See Figure \ref{fig:lemma},(1) for example.) Hence $R'$ and $R_{m_l}$ are connected via a shortest sequence of indexed regions $S=\{R_{a_{r+1}},\ldots, R_{a_u}=R_{m_j}, R_{b_1},\ldots,R_{b_v}\}$ in $H_j^+$ that contains $R_{m_j}$, where each $R_{b_k}$ is an indexed region with supporting hyperplanes $H_1,\ldots,H_s$. 

Note that the set of points in $H_j$ that are encompassed by $S$ has two $(d-1)$-dimensional components whose closures are disjoint. Consider the arrangement $\mathcal{A}'$ in $H_j$ cut out by all other hyperplanes in $\mathcal{A}$. We say that a region in $\mathcal{A}'$ is \textit{colored} if it is the common facet of two indexed region of $\mathcal{A}$, one in $H_j^+$ and the other in $H_j^-$. The union of this sequence of indexed regions intersects with $H_j$ at two totally disjoint colored regions. Property (b) says that the two regions have to be connected by a sequence of other colored regions of $\mathcal{A}'$ . Since the facet of $R_n$ that is contained in $H_j$ is not yet colored, there exists some $\tilde{M}\subset H_j$, $\tilde{M}\in L(\mathcal{A}')$ with $\dim(\tilde{M})\geq 1$ that does not satisfy the inductive hypothesis $(b)$ for $n-1$, which is a contradiction.

\textbf{Subcase (2).} If there exists a hyperplane $H_{j}\in\{H'_k, 1\leq k\leq t\}\cap\{H_{i}, 1\leq i\leq s\}$, then we can find a region $R'_k$  such that $R_{m_j}$, $R'=R'_k$, and $R_n$ are on the same side of $H_{j}$. Note that in the inductive case $n-1$, $H_{j}$ is a supporting hyperplane of $R'$ and $R_n$ is an unindexed region between $R_{m_l}$ and $R'$ (see Figure \ref{fig:lemma}.(ii) for instance), which points us back to the second half of the proof of case (1). 

Therefore, we conclude that $E_n\subseteq\{M\}.$

\

Then we want to show that the four properties remain true after adding $R_n$. For property (c), suppose that $R_n$ is the first to encompass $M\in E_n$. If $R_n$ is not the first indexed region in $K_s$, then we can find  $R_j, 1\leq j\leq n-1$  that is an indexed region in $K_s$  separated from $M$ by the fewest number of hyperplanes before we index $R_n$. If $R_j$ is strictly inside the interior of $K_s$, then we can apply the argument in part (1) of (ii) above by replacing $R'$ with $R_j$. Otherwise, use the argument in (ii).(2) by replacing $R'_k$ with $R_j$. Both lead to a contradiction to the induction hypothesis (b) of the case $n-1$. Therefore $R_n$ has to be the first indexed region in $K_s$, so property (c) holds for $n$. Properties (a), (b) and (d) follow immediately.
This completes the inductive step.
\end{proof}
Note that after adding all regions in $\mathcal{R(\mathcal{A})}$, we encompass all elements of $ L(\mathcal{A})$. Since $\mathcal{A}$ is semigeneral, we have $|L(\mathcal{A})|=|\mathcal{R}(\mathcal{A})|$. Therefore we need to encompass exactly one new element after indexing a new region, meaning that $|E_n|=1$ and $| E(\mathcal{B}^{(n)})|= n$ for all $n\in I$ as desired.
\end{proof}
From now on, we fix an indexing of $\mathcal{R(\mathcal{A})}$ that satisfies all the properties in Lemma \ref{lemmaexist}.
In order to compute the diagonal form of the Varchenko matrix of $\mathcal{A}$ , we need  the following two constructions:

\begin{definition}
Define $\varphi:\mathbb{Z}[x_1,\ldots,x_N]\rightarrow \mathbb{Z}[x_1,\ldots,x_N]$ to be the function satisfying the following properties:

(a) $\varphi(p+q)=\varphi(p)+\varphi(q)$ for all $p,q\in\mathbb{Z}[x_1,\ldots,x_N]$.
  
(b) $\varphi(p\cdot q)=\varphi(p)\varphi(q)$ for all monomials $p,q\in\mathbb{Z}[x_1,\ldots,x_N]$ with $\gcd(p,q)=1$.
  
(c) $\varphi(x^k_a)=x^2_a$ if $k\geq2$ and $\varphi(x_a^k)=x_a^k$ if $k=0,1$ for all $a=1,\ldots,N,$
   
(d) $\varphi(0)=0.$
\end{definition}   
   
It is easy to check that $\varphi$ is well-defined and unique. In fact, $\varphi(p)$ is obtained from $p$ by replacing all exponents $e\geq 3$ by $2$.

\begin{proposition} \label{prop34}
For all $i\in \{1,2, \ldots ,N\}$ and $p\in\mathbb{Z}[x_1,\ldots,x_{a-1}, x_{a+1},\ldots, x_N]$, 
 
$a)\ \varphi\big(x^2_a(1-x^2_a\cdot p)\big)= x^2_a\cdot \varphi(1-p);$
   
$b)\ \varphi\big((1-x^2_a)(1-x^2_a\cdot p)\big)=1-x^2_a$.  
\end{proposition}
\begin{proof}
The above identities follow directly from the definition of $\varphi$. 
\end{proof}

\begin{definition} \label{def:operate}
Let $P=[P_{i,j}]$ be an $N\times N$ symmetric matrix with entries $P_{i,j}\in\mathbb{Z}[x_1,\ldots,x_r]$. If $P_{k,k}$ is a factor of $P_{k,n}$, denoted as $P_{k,k} \big| P_{k,n}$, for all $n=1,\ldots,N$, we can define a matrix operation $T^{(k)}$ by 
  $$(T^{(k)}P)_{m,n}= \left\{ \begin{array}{rcl}
P_{k,k}\ \ \ \ \ \ \ \ \ \ \ \ \ \ \ \   & \mathrm{if\ } m=k,n=k& \\
P_{m,n}-\frac{P_{m,k}\cdot P_{n,k}}{P_{k,k}} & \mbox{otherwise} & \ 
\end{array}\right..$$

\end{definition}
  
In other words, we can apply $T^{(k)}$ to $P$ only when $P_{k,k}\big| P_{k,i}$ for all $i$. For each $i\neq k$, we subtract row $i$ by $\frac{P_{k,i}}{P_{k,k}}$ times row $k$ to get $P'$. After operating on all rows, we then subtract column $j$ of $P'$ by $\frac{P_{k,j}}{P_{k,k}}$ times column $k$ of $P'$ to get $T^{(k)}P$. It is easy to check that $T^{(k)}$ is a well-defined operation. The resulting matrix is also symmetric and the entries $P_{i,k}=P_{k,i}=0$ for all $i\neq k$.
  
\
  
Set $V^{(0)}=V(\mathcal{A})$. For all $k=1,\ldots,r(\mathcal{A})$, let $V^{(k)}$ be the matrix obtained by applying $T^{(1)},...,T^{(k)}$ in order as such to $V^{(0)}$. Denote the $(m,n)$-entry of $V^{(k)}$ by $V^{(k)}_{m,n}$. Suppose that  $E_k=\{M = H_{a_1}\cap H_{a_2}\cap\cdots\cap H_{a_s}\}$. Set $A_k=\{a_1,a_2,\ldots,a_s\}$.\label{def:postoperation}

\begin{lemma} \label{claim}
$a)\ \displaystyle{V^{(k)}_{k,k}=\prod_{a\in A_k}(1-x^2_{a})}; $

$b)\ V^{(k)}_{m,n}=0$ for all $m\neq n \leq k;$ $V^{(k)}_{m,m}=V^{(m)}_{m,m}$ for all $m\leq k$;

$c)\ \displaystyle{V^{(k)}_{m,n}=V_{m,n}\cdot \varphi\big(\prod_{i=1}^{k}(1-l^2_{i}(m,n))\big)}$ if at least one of $m, n$ is greater than $k$. 
\end{lemma}

\begin{proof} 
We will justify the Lemma by induction on $k$.
The statements hold for the base case $k=0$ by the definition of $V^{(0)}$.
  
\
  
Suppose that the statements hold for $k-1$.
  
It follows from Lemma \ref{lemmaexist}.c that $R_k$ has the smallest index in the polyhedron $K_{a_{1}a_{2}\cdots a_{s}}.$ For all $ 1\leq i\leq s,$ there exists a unique $r_i$ such that $ 1\leq r_i\leq k-1$ and $H_{a_i}$ connects $R_{r_i}$ and $R_k$, i.e., $\mathrm{sep}(R_{r_i},R_k)=\{H_{a_i}\}$. Furthermore, $r_i\neq r_j$ for all $i\neq j$. 
  
\ 
 
\begin{remark} $ V^{(k-1)}_{m,k}=V_{m,k}\cdot V^{(k-1)}_{k,k}$ or $0$.\label{remark}
\begin{proof}   
If $R_m$ is not in  $K_{a_{1}a_{2}\cdots a_{s}}$, then $R_m$ and $R_k$ are on different sides of at least one of the hyperplanes $H_{a_1},\ldots,H_{a_s}$, say $H_{a_1}$. Thus $H_{a_1} \notin \mathrm{sep}(R_{r_1},R_m)$, i.e., $\mathrm{sep}(R_{r_1},R_m)\cap \mathrm{sep}(R_{r_1},R_k)=\mathrm{sep}(R_{r_1},R_m)\cap \{H_{a_1}\}=\emptyset.$ Therefore $l_{r_1}(m,k)=1.$ By the induction hypothesis, we have $$V^{(k-1)}_{m,k}=V_{m,k} \cdot \varphi\big((1-l^2_{r_1}(m,k))\cdot\prod_{j=1,j\neq r_1}^{k-1}(1-l^2_{j}(m,k))\big)=0.$$ 

If $R_m$ is contained in $K_{a_{1}a_{2}\cdots a_{s}}$, then $R_m$ and $R_k$ are on the same side of $H_{a_i}$ for all $1\leq i \leq s$. Therefore, for any indexed region $R_j$, $j=1,2,\ldots,k-1$, at least one of $H_{a_1},\ldots, H_{a_s}$ separates $R_j$ and $R_m\cup R_k$, say $H_{a_l}$. Since $H_{a_l}\in \mathrm{sep}(R_j,R_m)$, it follows that $x_{a_l} \mid l_{j}(m,k).$ Note that $H_{a_i} \in \mathrm{sep}(R_{r_i},R_m)$ for all $1 \leq i \leq s,$ so $\mathrm{sep}(R_{r_i},R_m) \cap \mathrm{sep}(R_{r_i},R_k) = \mathrm{sep}(R_{r_i},R_m) \cap \{H_{a_i}\}=\{H_{a_i}\}$. Therefore $ l_{r_i}(m,k)=x_{a_i}.$ Applying the results of Proposition \ref{prop34}.b, we get
\begin{align*}
V^{(k-1)}_{m,k}=&V_{m,k} \cdot \varphi\big((1-l^2_{r_1}(m,k))\cdots(1-l^2_{r_s}(m,k))\cdot\prod_{\stackrel{j=1,}{j \neq r_1,\ldots ,r_s}}^{k-1}(1-l^2_{j}(m,k))\big)\\
=&V_{m,k} \cdot (1-x^2_{a_1})(1-x^2_{a_2})\cdots(1-x^2_{a_s}),
\end{align*}
since $R_j\nsubseteq K_{a_{1}a_{2}\cdots a_{s}} $ is separated from $R_m$ by at least one of $H_{a_1},\ldots, H_{a_s}$.

On the other hand, by the induction hypothesis 
\begin{align*}
&V^{(k-1)}_{k,k}=V_{k,k} \cdot \varphi\big(\prod_{j=1}^{k-1}(1-l^2_{j}(k,k)\big)\\
=&\varphi\big((1-l^2_{r_1}(k,k))\cdots(1-l^2_{r_s}(k,k))\cdot\prod_{\stackrel{j=1,}{j \neq r_1,\ldots ,r_s}}^{k-1}(1-l^2_{j}(k,k))\big)\\
=&(1-x^2_{a_1})(1-x^2_{a_2})\cdots(1-x^2_{a_s}). \numberthis
\end{align*}
Hence we conclude that $$V^{(k-1)}_{m,k}=V_{m,k} \cdot V^{(k-1)}_{k,k}.$$  
\end {proof}
\end{remark}
  
Since $V^{(k-1)}_{k,k} \mid V^{(k-1)}_{m,k}$ for all $m=1,2,\ldots,r(\mathcal{A})$,  we can apply the matrix operation $T^{(k)}$ to $V^{(k-1)}$.
By definition of $T^{(k)}$, if $m\neq n \leq k$, then $V^{(k)}_{m,n}=0$. Otherwise,
\begin{align*}
V^{(k)}_{m,n}=V^{(k-1)}_{m,n}-\frac{V^{(k-1)}_{m,k}\cdot V^{(k-1)}_{n,k}}{V^{(k-1)}_{k,k}}.\numberthis
\end{align*}
 
It follows immediately that $\displaystyle{V^{(k)}_{k,k}=V^{(k-1)}_{k,k}=(1-x^2_{a_1})(1-x^2_{a_2})\cdots(1-x^2_{a_s})}$. Therefore, claim (a) holds for $k$.
In addition, we can deduce from Remark \ref{remark} that if at least one of $R_m, R_n$ is not contained in $K_{a_{1}a_{2}\cdots a_{s}}$, then
\begin{align*}
V^{(k)}_{m,n}=V^{(k-1)}_{m,n}-\frac{V^{(k-1)}_{m,k}\cdot V^{(k-1)}_{n,k}}{V^{(k-1)}_{k,k}}=V^{(k-1)}_{m,n}-0=&{V_{m,n}\cdot \varphi\big(\prod_{i=1}^{k-1}(1-l^2_{i}(m,n)\big)}\\
=&{V_{m,n}\cdot \varphi\big(\prod_{i=1}^{k}(1-l^2_{i}(m,n)\big)}.
\end{align*}
Note that if $m\neq n \leq k$, then neither of $R_m,R_n$ is contained in $K_{a_{1}a_{2}\cdots a_{s}}$ and $ V^{(k)}_{m,n}=V^{(k-1)}_{m,n}=0$  by the induction hypothesis. If $m=n<k$, then $V^{(k)}_{m,m}=V^{(k-1)}_{m,m}=\cdots=V^{(m)}_{m,m}$. Hence (b) also holds for $k$.
   
\   

In order to prove (c), it suffices to show that if $m,n$ are both contained in $K_{a_{1}a_{2}\cdots a_{s}}$, then  
$${V_{m,n}\cdot \varphi\big(\prod_{i=1}^{k-1}(1-l^2_{i}(m,n)\big)}-V_{m,n}\cdot \varphi\big(\prod_{i=1}^{k}(1-l^2_{i}(m,n)\big)= V^{(k-1)}_{m,n}-V^{(k)}_{m,n}$$
\begin{align} 
=\frac{V^{(k-1)}_{m,k}\cdot V^{(k-1)}_{n,k}}{V^{(k-1)}_{k,k}}
=V_{m,k}\cdot V_{n,k}\cdot V^{(k-1)}_{k,k}
=V_{m,n}\cdot l^2_{k}(m,n) \cdot V^{(k-1)}_{k,k}.
\end{align}

Using linearity of $\varphi$, we can combine the two terms on the left hand side of the above equation:
\begin{equation}
LHS=V_{m,n}\cdot \varphi\big(l^2_{k}(m,n)\cdot\prod_{i=1}^{k-1}(1-l^2_{i}(m,n)\big).
\end{equation}
So we only need to show that 
\begin{equation}
\varphi\big(l^2_{k}(m,n)\cdot\prod_{i=1}^{k-1}(1-l^2_{i}(m,n))\big)=l^2_{k}(m,n) \cdot V^{(k-1)}_{k,k}.
\end{equation}
It follows from Proposition \ref{prop34}.(a) that $$l^2_{k}(m,n) \mid \varphi\big(l^2_{k}(m,n)\cdot\prod_{i=1}^{k-1}(1-l^2_{i}(m,n))\big);\ l^2_{k}(m,n) \mid\varphi\big(l^2_{k}(m,n) \cdot l^2_{i}(m,n)\big).$$ Hence we can pull out $l^2_{k}(m,n)$ on the left hand side:
\begin{align}
\varphi\big(l^2_{k}(m,n)\cdot\prod_{i=1}^{k-1}(1-l^2_{i}(m,n))\big)=l^2_{k}(m,n) \cdot \varphi\big(\prod_{i=1}^{k-1}(1-\tilde{l}^2_{i}(m,n))\big),\\
\textrm{where } \tilde{l}^2_{i}(m,n)=\frac{\varphi\big(l^2_{k}(m,n)\cdot l^2_{i}(m,n)\big)}{l^2_{k}(m,n)}\textrm{ for all }i =1,2, \ldots, k-1.
\end{align}

Note that $\tilde{l}_{i}(m,n)$ is exactly the \textit{distance} between $R_i$ and $R_m\cup R_n$ in $\tilde{\mathcal{A}}_{m,n}=\mathcal{A} \backslash \big(\mathrm{sep}(R_k,R_m)\cap \mathrm{sep}(R_k,R_n)\big)=\mathcal{A} \backslash \mathrm{sep}(R_k,R_m\cup R_n).$
Therefore it suffices to show that in $\tilde{\mathcal{A}}_{m,n},$
\begin{equation}
\varphi\big(\prod_{i=1}^{k-1}(1-\tilde{l}^2_{i}(m,n))\big)=V^{(k-1)}_{k,k}.
\end{equation}
   
Since $R_m, R_n, R_k$ are still contained in the cone formed by $H_{a_1},\ldots,H_{a_s}$ in $\tilde{\mathcal{A}}$, it is clear that

\begin{align*}
 \varphi\big(\prod_{j=1}^{k-1}(1-\tilde{l}^2_{j}(m,n))\big)
 =&\varphi\big((1-\tilde{l}^2_{r_1}(m,n))\cdots (1-\tilde{l}^2_{r_s}(m,n))\cdot\prod_{\stackrel{j=1,}{j\neq r_1,\ldots ,r_s}}^{k-1}(1-\tilde{l}^2_{j}(m,n))\big)\\
 =&(1-x^2_{r_1})\cdots(1-x^2_{r_s})=V^{(k-1)}_{k,k}.
\end{align*}
 Hence we conclude that (c) holds for $k$, and this completes the induction. 
\end{proof}

\ 

An immediate corollary of Lemma \ref{claim} is that when $k=r(\mathcal{A}),\ V^{r(\mathcal{A})}_{m,n}=0$ for all $m\neq n$. Thus we have reduced $V$ to a diagonal matrix. We know from Lemma \ref{lemmaexist}.c that by adding region $k$ we encompass exactly one new element $M\in L(\mathcal{A})$, so each entry $\displaystyle{V^{(k)}_{k,k}=\prod_{a\in A_k}(1-x^2_{a})}$  appears exactly once on the diagonal of $V^{r(\mathcal{A})}$. Hence we've proven the sufficient condition of Theorem \ref{main}.
 
 \

\section{Nonexistence of the Diagonal Form of the Varchenko Matrices of Arrangements Not in Semigeneral Positions}
\label{section 4}
In this section, we will prove the necessary condition of Theorem \ref{main}.
  
\begin{lemma} \label{lem:sub}
Suppose that $\mathcal{A}$ is a finite hyperplane arrangement in $\mathbb{R}^d$ and $H\notin\mathcal{A}$ is a $(d-1)$-dimensional hyperplane. If $V(\mathcal{A}\cup\{H\})$ has a diagonal form over $\mathbb{Z}[x_1,\ldots,x_n]$, then $V(\mathcal{A})$ also has a diagonal form over $\mathbb{Z}[x_1,\ldots,x_n].$
\end{lemma}
  
\begin{proof}
Let $x_1,\ldots,x_n$ be the indeterminates for hyperplanes in $\mathcal{A}$ and $x_{n+1}$ for $H$. Set $V=V(\mathcal{A})$ and $V^{(0)}=V(\mathcal{A}\cup\{H\})$.
  
Let $V^{(1)}$ be the matrix obtained by setting $x_{n+1}=1$ in $V^{(0)}$. Observe that the $i^{th}$ row (column) and the $j^{th}$ row (column) of $V^{(1)}$ is the same for all $i\neq j$ if $V^{(0)}_{i,j}=x_{n+1}$, i.e., region $i$ and $j$ are separated only by $S$. Apply row and column operations to eliminate repeated rows (columns), and we will get $V^{(1)}\sim V\oplus \textbf{0}_{k}$, where $\textbf{0}_k$ is the all zero matrix of dimension $k\times k$ and $k=r(\mathcal{A}\cup\{H\})-r(\mathcal{A})$.
  
If $V^{(0)}$ has a diagonal form over $\mathbb{Z}[x_1,\ldots,x_{n+1}]$, then we can assign an integer value to $x_{n+1}$ and hence $V^{(1)}$ and $V\oplus\textbf{0}_k$ have a diagonal form over $\mathbb{Z}[x_1,\ldots,x_n]$.
  
Let $D$ be the diagonal form of $V\oplus\textbf{0}_k$. According to Theorem \ref{det}, $\det(V)\neq0$; by Lemma \ref{lemmasimilar1}, $\mathrm{rank}(D)=\mathrm{rank}(V^{(1)})$, which is equal to the dimension $r(\mathcal{A})$ of $V$. Therefore, the number of zeros on $D$'s diagonal is $k$.
  
Note that there exist matrices $P,Q$ of dimension $r(\mathcal{A}\cup\{H\})$ and unit determinant such that $P(V\oplus \textbf{0}_k)=DQ.$ We can also write the matrices in the following way, where $D'$ is the diagonal matrix obtained from eliminating the all zero rows and columns in $D$:
  \[
\left[
\begin{array}{ccc|c}
& & & \\
& P_1 & & P_3 \\
& & & \\ \hline
& P_2 & & P_4 \\
\end{array}
\right]
\left[
\begin{array}{ccc|c}
& & & \\
& V & & 0 \\
& & & \\ \hline
& 0 & &0 \\
\end{array}
\right]
=
\left[
\begin{array}{ccc|c}
& & & \\
& D' & & 0 \\
& & & \\ \hline
& 0 & &0 \\
\end{array}
\right]
\left[
\begin{array}{ccc|c}
& & & \\
& Q_1 & & Q_2 \\
& & & \\ \hline
& Q_3 & & Q_4 \\
\end{array}
\right];
\]

  \[
\left[
\begin{array}{ccc|c}
& & & \\
& P_1V & & 0 \\
& & & \\ \hline
& P_2V & & 0 \\
\end{array}
\right]
=
\left[
\begin{array}{ccc|c}
& & & \\
& D'Q_1 & & D'Q_2 \\
& & & \\ \hline
& 0 & &0 \\
\end{array}
\right].
\]

It is easy to check that $P_2V=0$, $D'Q_2$=0, $P_1V=D'Q_1$.

Since $\det(V)\neq0$ and $\det(D')\neq0$, $P_2$ and $Q_2$ have only 0 entries. Therefore, $1=\det(P)=\det(P_1)\det(P_4)$; $1=\det(Q)=\det(Q_1)\det(Q_4)$. The only units in $\mathbb{Z}[x_1,\ldots,x_n]$ are $1$ and $-1$, so we can assume that $\det(P_1)=\det(Q_1)=1$. Thus, $D'$ is a diagonal form of $V$. 
\end{proof}

Now we've arrived at the main theorem of this section, which is also the necessary condition of Theorem \ref{main}.

\begin{theorem}
Let $\mathcal{A}$ be a hyperplane arrangement in $\mathbb{R}^d$ that is not semigeneral. Then $V(\mathcal{A})$ does not have a diagonal form over $\mathbb{Z}[x_1,\ldots,x_n]$.
\end{theorem}

\begin{proof}
Using Lemma \ref{lem:sub}, we can delete as many hyperplanes in $\mathcal{A}$ as possible so that the resulting arrangement, denoted again as $\mathcal{A}$, satisfies the nonsemigeneral property and the \textit{minimum condition}, i.e., if we delete any hyperplane, the remaining arrangement will be semigeneral.

Note that there must exist $H_1,\ldots,H_p\in\mathcal{A}$ with nonempty intersection such that $\dim({H_1\cap\cdots\cap H_p})\neq d-p$ . Hence $\dim({H_1\cap\cdots\cap H_p})\geq d-p+1$. If $\dim({H_1\cap\cdots\cap H_p})\geq d-p+2$, then $\dim({H_2\cap\cdots\cap H_p})\geq d-(p-1)$, contradicting the minimum condition of $\mathcal{A}$. Therefore,  $\dim({H_1\cap\cdots\cap H_p})=d-p+1$. Also by the minimum condition, $\mathcal{A}=\{H_1,\ldots,H_p\}$.

Without loss of generality, we can assume that $p=d+1$ (by projecting all hyperplanes to a smaller subspace). Thus we only need to consider the case of $\mathcal{A}=\{H_1,\ldots,H_{d+1}\}$ in $\mathbb{R}^d$, where the intersection $H_1\cap\cdots\cap H_{d+1}$ is a single point and any hyperplane arrangement formed by a subset of $\mathcal{A}$ with cardinality $d$ is semigeneral and nonempty.

Now we can deduce the structure of the intersection poset $L(\mathcal{A})$: $L(\mathcal{A})$ consists of the intersection of any $k$ hyperplanes in $\mathcal{A}$ for all $k=0,1,\ldots,d-1$ and the point $H_1\cap\cdots \cap H_{d+1}$, which is also the intersection of any $d$ hyperplanes in $\mathcal{A}$. It follows that for all $H\in\mathcal{A}$, $\mathcal{A}^H$ is the hyperplane arrangement of $d$ hyperplanes of dimension $d-1$ intersecting at one point in $\mathbb{R}^{d-1}$, so $r(\mathcal{A}^S)=2^d-2, \ r(\mathcal{A})=2^{d+1}-2.$

Let $x_1,\ldots,x_{d+1}$ be the indeterminates of the hyperplanes in $\mathcal{A}$. By Theorem \ref{det}, $\det V(\mathcal{A})=(1-x_1^2)^{2^d-2}(1-x_2^2)^{2^d-2}\cdots(1-x_{d+1}^2)^{2^d-2}(1-x_1^2\cdots x_{d+1}^2)$.

Set $H_1\cap\cdots\cap H_{d+1}$ as the origin $(0,0,\ldots,0)$. Pick any hyperplane $H\in\mathcal{A}.$ Then $S$ separates the space into two half spaces $H_1, H_2$. Since $\mathcal{A}$ is symmetric about the origin, exactly half of $\mathcal{R}(\mathcal{A})$ lies in $H_1$, i.e., $r(H_1)=\frac{1}{2}r(\mathcal{A})=2^{d}-1$. We also know that $r(\mathcal{A}^S)=2^d-2=(2^d-1)-1$. Thus, for all but one region $R$ in $H_1$, the intersection of its closure and $H$ has dimension $d-1$. The intersection of the closure of $R$ with $H$ is the point $H_1\cap\cdots\cap H_{d+1}$.

Here comes an important observation: If we restrict to the Varchenko matrix of regions in $H_1$, denoted by $\mathcal{R}(H_1)$, we obtain the Varchenko matrix of the hyperplane arrangement of $d$ hyperplanes in $\mathbb{R}^{d-1}$ in general position. This matrix is equivalent to the Varchenko matrix for $\mathcal{A}^H\cup R$. Intuitively, we can view $R$ as the inner region of the point $H_1\cap\cdots\cap H_{d+1}$.

\

We'll prove the theorem by contradiction. Suppose that $D$ is a diagonal form of $V(\mathcal{A})$, then $V\sim D$.

First, we want to show that in $D$'s diagonal entries, $1-x_i$ and $1+x_i$ must appear in the form of $1-x_i^2$ for all $i=1,\ldots,d+1$.

For all $i=1,\ldots,d+1$, consider $V_{x_i=3,x_j=0\ \forall j\neq i}$. It can be decomposed into blocks of $\begin{bmatrix}1&3\\3&1\end{bmatrix}$ and identity matrices, so its SNF has diagonal entries 1 and 8 with multiplicities. Now $D_{x_i=3,x_j=0\ \forall j\neq i}$ has the same SNF by Lemma \ref{lemmasimilar2}. Note that $1-x_i=-2$, $1+x_i=4$, and their products are all powers of 2, so $D_{x_i=3,x_j=0\ \forall j\neq i}$ is already in its SNF. We have equal numbers of $4$ and $-2$, so they must pair up in the form of $1-x_i^2$ or we won't have only 1 and 8 on the diagonal. In addition, $1-x_i^2$ can appear at most once in each diagonal entry of $D$.

We can ignore the terms $1-x_1\cdots x_{d+1}$ and $1+x_1\cdots x_{d+1}$ for the moment since we will assign at least one of $x_i$ to 0 in the following steps.

If we set $x_{d+1}=0$, we will get two blocks of matrices corresponding to a general position using the earlier observation. If we set all other indeterminates equal to the indeterminate $q$, the diagonal form of $V_{x_{d+1}=0}$ has diagonal entries $(1-x_{i_1}^2)\cdots(1-x_{i_k}^2)$ (twice) for all $k=0,\ldots,d-1$, $1\leq i_1<\cdots<i_k\leq d$. Thus, the SNF of $V_{x_i=0,x_j=q,\forall j\neq i}$ (over $\mathbb{Z}[q]$) has diagonal entries $(1-q^2)^k$ ($2{d\choose k}$ times) for all $k=0,1,\ldots,d-1$.

We call a diagonal entry of $D$ a \textit{$k-$entry} if after setting $x_1=\ldots=x_{d+1}=q$, it becomes $(1-q^2)^k$. All diagonal entries of $D_{x_i=0,x_j=q,\forall j\neq i}$ have the form $(1-q^2)^k$ for some $k$ so it is already in its SNF.  Since the SNF is unique, $D_{x_i=0,x_j=q,\forall j\neq i}$ has diagonal entries $(1-q^2)^k$ ($2{d\choose k}$ times) for all $k=0,1,\ldots,d-1$.

We then compute the exact number of $k-$entries in $D$. The number of $k-$entries in $D$ is 0 if $k\geq d$; otherwise, there exists some $i\in\{1,\ldots,d+1\}$ such that $D_{x_i=0,x_j=q,\forall j\neq i}$ has a diagonal entry $(1-q^2)^d$, which leads to a contradiction. 

\begin{claim}The number of $(d-2k-1)$-entries in $D$ is $2{d+1\choose 2k+2}$ and the number of $(d-2k-2)$-entries in $D$ is 0. 
\begin{proof}
We'll prove the claim by induction on $k$. It is true when $k=0$.

If the claim holds for $k$,  i.e., the number of $(d-2k)$-entries in $D$ is 0.
Assume that the number of $(d-2k-1)$-entries in $D$ is $m$ and $1-x_i^2$ appears in $a_i$ of these entries. 

Set $x_i=0$ and all other indeterminates equal to $q$. Since there are  exactly $2{d\choose d-2k-1}$ number of $(1-q^2)^{d-2k-1}$'s and no $(d-2k)$-entries in $D$,  
$m-a_i=2{d\choose d-2k-1}.$
Therefore $a_i$ is a constant with respect to $i$.

By a simple double counting of the total number of $(1-\Box^2)$'s in those entries, where $\Box=x_1,\ldots,x_{d+1}$, we have
$$\sum_{i=1}^{d+1}a_i=\Bigg(m-2{d\choose d-2k-1}\Bigg)\cdot(d+1)=m\cdot(d-2k-1).$$
It is easy to check that $m=2{d+1\choose 2k+2}$ and $a_i=2{d+1\choose 2k+2}-2{d\choose d-2k-1}=2{d\choose 2k+2}$ for all $i=1,\ldots,d+1$.

Now if we assign 0 to $x_i$ and $q$ to all other indeterminates, we already have $(1-q^2)^{d-2k-2}$ occuring $a_i=2{d\choose 2k+2}$ times so we do not need more. Therefore, all the $(d-2k-2)$-entries in $D$ must contain $1-x_i^2$. It is true for all $i=1,\ldots,d+1$ so we have a contradiction unless the number of $(d-2k-2)$-entries is 0.

This completes the induction and proves the claim. 
\end{proof}
\end{claim}

If $d$ is even, the number of 0-entries (1's) is 0 in $D$. In other words, if we assign 1 to all indeterminates, $D$ becomes the all zero matrix with rank 0, while the rank of $V$ becomes 1, which is a contradiction  by Lemma \ref{lemmasimilar1}. 

\

If $d$ is odd, we have to take into account the terms $1-x_1\cdots x_{d+1}$ and $1+x_1\cdots x_{d+1}$. As before, we will first show that they must pair up.

Let $x_1=3$, $x_i=1$ for all $i\geq2$ in $V$. Then in $V$, row $i$ is identical to row $j$ if region $i$ and $j$ are on the same side of $H_1$. Eliminating repeated rows with row and column operations , we get $\begin{bmatrix}1&3\\3&1\end{bmatrix}\oplus \textbf{0}$ where $\textbf{0}$ is the all zero matrix, which has SNF $\{1,8,0$ (multiple times)$\}$. Note that when $d$ is odd, there are no $1-$entries. Thus 1 and 8 must come from a combination of $1-x_1\cdots x_{d+1}$ and $1+x_1\cdots x_{d+1}$. Hence they must appear in the form of $1-x_1^2\cdots x_{d+1}^2$.

Furthermore, $1-x_1^2\cdots x_{d+1}^2$ must appear alone in a $0$-entry. Otherwise, since there are no $1-$entries in $D$, after assigning $x_i=1$ for $i\geq2$ we will end up with a matrix with only two 1's on the diagonal and 0 everywhere else. Since $d$ is odd, the number of $0-$entries is 2. One of them is $1-x_1^2\cdots x_{d+1}^2$ and the other one can only be a true 1. 

Consider $V_{x_1=\cdots=x_{d+1}=3}$ and $D_{x_1=\cdots=x_{d+1}=3}.$ Since $V_{x_1=\cdots=x_{d+1}=3}$ has a submatrix $\begin{bmatrix}1&3\\3&1\end{bmatrix}$, we deduce that
$\gcd(V_{x_1=\cdots=x_{d+1}=3},2)\leq 8$. On the diagonal of $D_{x_1=\cdots=x_{d+1}=3}$, there is a 1, a $1-3^{2d+2}$, and all other entries are multiples of $(1-3^2)^2=64$ since there is no 1-entry. 

Note that $1-3^{2d+2}=(1-3^2)(1+3^2+3^4+\cdots+3^{2d})$. Since $d$ is odd, so $1+3^2+3^4+\cdots+3^{2d}$ is even and $16\ |\ (1-3^{2d+2})$. 
Therefore, $16\ |\ \gcd(D_{x_1=\cdots=x_{d+1}=3},2)$, which leads to a contradiction by Lemma \ref{lemmasimilar1}.

Hence we conclude that $V(\mathcal{A})$ does not have a diagonal form. 
\end{proof}

\section{Acknowledgment}
This research was conducted during the 2015 summer UROP (Undergraduate Research Opportunities Program) at MIT . The authors are very grateful to Prof. Richard P. Stanley for suggesting the topic and supervising the entire project.

\

This is a pre-print of an article published in Journal of Algebraic Combinatorics. The final authenticated version is available online at: \url{https://doi.org/10.1007/s10801-018-0813-7}.

\end{document}